\documentclass[11pt]{amsart}

\usepackage[english]{babel} 
\usepackage[utf8]{inputenc}
\usepackage{cmbright}
\usepackage[T1]{fontenc}
\usepackage{lmodern} 
\usepackage{soul,xfrac}
\usepackage{mathtools, textcomp}
\mathtoolsset{showonlyrefs} 
\usepackage[shortlabels]{enumitem}
\usepackage{cite} 
\usepackage{amsmath,amssymb,amsthm,mathtools,color}
\usepackage{xcolor}
\usepackage{hyperref}
\definecolor{darkgreen}{rgb}{0,0.4,0}
\definecolor{BrickRed}{rgb}{0.65,0.08,0}
\hypersetup{colorlinks=true,linkcolor=blue,citecolor=red,filecolor=BrickRed,urlcolor=darkgreen}

\usepackage{graphicx} 
\graphicspath{{pics/}}

\newtheorem{thm}{THEOREM}

\newtheorem{rmk}{REMARK}

\newtheorem{lemma}[thm]{Lemma}

\numberwithin{equation}{section}

\def\N{\mathbb{N}}

\def\F{\mathbb{F}}

\def\ainFq*{a\in \F_q^*}

\def\rinN*{r \in \N^*}

\def\ninN*{n \in \N^*}

\def\SGA412{\text{SGA} 4 \sfrac{1}{2}}

\addto\captionsenglish{
	\renewcommand{\contentsname}%
	{TABLE OF CONTENTS}%
}

\begin{document}
\title{On Ramanujan-Fourier expansions} 


\author{David T. Nguyen}
\email{d.nguyen@queensu.ca}
\address{Dept. of Math. and Stats., Queens' University, Kingston}

\newcount\m \newcount\n
\def\hours{\n=\time \divide\n 60
	\m=-\n \multiply\m 60 \advance\m \time
	\twodigits\n:\twodigits\m}
\def\twodigits#1{\ifnum #1<10 0\fi \number#1}
\date{\today}

\maketitle

We cut to the chase. In this short note, we aim to do the following two things.
\begin{enumerate}
	\item We give, yet, another heuristic derivation for the general shifted divisor correlation leading order term, valid for any $k,\ell \ge 1$ and composite $h$:
	\begin{equation} \label{eq:main}
		\sum_{ n\le X}
		\tau_k(n) \tau_\ell(n+h)
		\sim
		C_{k, \ell}
		f_{k,\ell}(h)
		X
		\frac{(\log X)^{k+ \ell - 2}}{(k-1)! (\ell -1)!},
		\quad
		(X \to \infty),
	\end{equation}
	where $\tau_k(n)$ is the $k$-fold divisor function, with
	\begin{equation} \label{eq:C}
		C_{k, \ell}
		=
		\prod_p
		\left(
		1
		-
		\left(
		1 - 
		\left(1- \frac{1}{p}\right)^{k-1}
		\right)
		\left(
		1 - 
		\left(1- \frac{1}{p}\right)^{\ell-1}
		\right)
		\right),
	\end{equation}
	\begin{equation} \label{eq:f}
		f_{k, \ell}(h)
		=
		\prod_{p \mid h}
		\left(
		\frac{
			1
			+ \displaystyle\sum_{1\le j \le \nu_p(h)}
			(p^j - p^{j-1})
			\frac{A_{k,\ell}(p^j)}{p^{2j}}
			-
			p^{\nu_p(h)}
			\frac{A_{k,\ell}(p^{\nu_p(h)+1})}{p^{2(\nu_p(h)+1)}}
		}{1
			-
			\left(
			1 - 
			\left(1- \frac{1}{p}\right)^{k-1}
			\right)
			\left(
			1 - 
			\left(1- \frac{1}{p}\right)^{\ell-1}
			\right)}
		\right),
	\end{equation}
	\begin{equation} \label{eq:Akl}
		A_{k,\ell}(p^j)
		=
		\left(1-\frac{1}{p}\right)^{k + \ell -2}
		\sum_{ \alpha \ge j}
		\binom{k + \alpha -2}{k-2}
		p^{j - \alpha}
		\sum_{ \beta \ge j}
		\binom{\ell + \beta -2}{\ell-2}
		p^{j - \beta},
	\end{equation}
	and	$\nu_p(h)$ the highest power of $p$ dividing $h$.

	\item We prove that our global constant $C_{k, \ell}$ above in \eqref{eq:C} and the local Euler factor $f_{k,\ell}(h)$ in \eqref{eq:f} match exactly that of Ng and Thom's coming from heuristic of the circle method, for all $k, \ell \ge 1$ and any composite shift $h$. The proof of these is simple.
\end{enumerate}

Our derivation gives a new arithmetic interpretation to the leading constant $C_{k, \ell}
f_{k,\ell}(h)$ in \eqref{eq:main}, and, is, of course, not rigorous (and it may be impossible to make it rigorous, because, if it were, it would also give the twin prime conjecture for free). It can be used, however, to make some precise predictions. The method is based on the (not-yet-well-developed) theory of Ramanujan-Fourier expansions, in which the exponential phase is replaced by Ramanujan sums.

\subsection{Definition}
Given an arithmetic function $f(n)$, we say $f(n)$ has an RF expansion if, for each $n$, it can be written in the form
\begin{equation}
	f(n) = 
	\sum_{q=1}^\infty
	\hat{f}(q) c_q(n)
\end{equation}
for some coefficients $\hat{f}(q)$ that make the above series convergent. Here $c_q(n)$ is the Ramanujan sum.


\subsection{Lemmas}
We need the following two lemmas.
\begin{lemma}[Lucht (1995) \cite{Lucht1995}[Prop. 3, p. 40]]
	One has the following conditionally convergent RF expansion:
	\begin{equation} \label{eq:Lucht}
		\tau_k(n)
		=
		\sum_{q=1}^\infty
		\hat{\tau_k}(q)
		c_q(n),
	\end{equation}
	where
	\begin{equation} \label{eq:rf}
		\hat{\tau_k}(q)
		=
		\frac{(-1)^{k-1}}{(k-1)!}
		\frac{\log^{k-1}(q)}{q}
		\prod_{p\mid q}
		\left(1-\frac{1}{p}\right)^{k-1}
		\sum_{ \alpha \ge \nu_p(q)}
		\binom{k + \alpha -2}{k-2}
		p^{\nu_p(q) - \alpha}.
	\end{equation}
\end{lemma}
\begin{rmk}
	It seems that \eqref{eq:rf} corrects a small sign error in the original stated result of Lucht (the $-1$ in the exponent in \cite{Lucht1995}[Prop. 3, p. 40] should not be there). We believe so because if we kept the sign from Lucht, it produces the wrong answer in \eqref{eq:main} when compared to \cite{NgThom2019}.
\end{rmk}

\begin{lemma}[Carmichael (1932) \cite{Carmichael1932}]
	We have the following orthogonality relation/correlation of two Ramanujan sums:
	\begin{equation} \label{eq:Carmichael}
		\sum_{n \le X}
		c_{q_1}(n)
		c_{q_2}(n+h)
		\sim
		\begin{cases}
			X c_q(h), & \text{ if } q_1=q_2=q,\\
			0, & \text{otherwise}.
		\end{cases}
	\end{equation}
\end{lemma}

\begin{rmk}
	A finer asymptotic of the left side of \eqref{eq:Carmichael} can be found in \cite[Lemma 2, p. 694]{GMP2014}.
\end{rmk}

\subsection{Derivation of (\ref{eq:main}) via RF expansions}
By \eqref{eq:Lucht}, we have
\begin{equation}
	\sum_{ n\le X}
	\tau_k(n) \tau_\ell(n+h)
	=
	\sum_{n\le X}
	\sum_{q_1=1}^\infty
	\hat{\tau_k}(q_1)
	c_{q_1}(n)
	\sum_{q_2=1}^\infty
	\hat{\tau_\ell}(q_2)
	c_{q_2}(n+h).
\end{equation}
Bringing the $n$-sum inside, we get, by \eqref{eq:Carmichael},
\begin{equation} \label{eq:941}
	\sum_{ n\le X}
	\tau_k(n) \tau_\ell(n+h)
	=
	\sum_{q_1=1}^\infty
	\hat{\tau_k}(q_1)
	\sum_{q_2=1}^\infty
	\hat{\tau_\ell}(q_2)
	\sum_{n\le X}
	c_{q_1}(n)
	c_{q_2}(n+h)
	\sim
	X
	\sum_{q=1}^\infty
	\hat{\tau_k}(q)
	\hat{\tau_\ell}(q)
	c_q(h).
\end{equation}
Replacing $(-1)^k \log^{k-1}(q)$ by $\log^{k-1} X$, the product $\hat{\tau_k}(q)
\hat{\tau_\ell}(q)$ becomes, by \eqref{eq:rf},
\begin{equation} \label{eq:1006}
	\frac{(\log X)^{k+ \ell - 2}}{(k-1)! (\ell -1)!}
	\frac{1}{q^2}
	A_{k,\ell}(q),
\end{equation}
where
\begin{equation} \label{eq:718}
	A_{k,\ell}(q)
	=
	\prod_{p\mid q}
	\left(1-\frac{1}{p}\right)^{k + \ell -2}
	\sum_{ \alpha \ge \nu_p(q)}
	\binom{k + \alpha -2}{k-2}
	p^{\nu_p(q) - \alpha}
	\sum_{ \beta \ge \nu_p(q)}
	\binom{\ell + \beta -2}{\ell-2}
	p^{\nu_p(q) - \beta}.
\end{equation}
In particular, we have
\begin{equation} \label{eq:Aklp}
	A_{k,\ell}(p)
	=
	p^2 
	\left(
	1 - 
	\left(1- \frac{1}{p}\right)^{k-1}
	\right)
	\left(
	1 - 
	\left(1- \frac{1}{p}\right)^{\ell-1}
	\right).
\end{equation}
Thus, by \eqref{eq:1006}, \eqref{eq:941} becomes
\begin{equation} \label{eq:1125b}
	\sum_{ n\le X}
	\tau_k(n) \tau_\ell(n+h)
	\sim
	X
	\frac{(\log X)^{k+ \ell - 2}}{(k-1)! (\ell -1)!}
	B_{k,\ell}(h),
\end{equation}
where
\begin{equation} \label{eq:712}
	B_{k,\ell}(h)
	=
	\sum_{q=1}^\infty
	\frac{c_q(h)}{q^2}
	A_{k,\ell}(q).
\end{equation}
Going to Euler products, we have
\begin{equation} \label{eq:1049}
	B_{k,\ell}(h)
	=
	\prod_{p \nmid h}
	\left(
	1
	+
	\sum_{j=1}^\infty
	\frac{c_{p^j}(h)}{p^{2j}}
	A_{k,\ell}(p^j)
	\right)
	\prod_{p \mid h}
	\left(
	1
	+
	\sum_{j=1}^\infty
	\frac{c_{p^j}(h)}{p^{2j}}
	A_{k,\ell}(p^j)
	\right),
\end{equation}
say. 

If $p \nmid h$, then
\begin{equation}
	c_{p^j}(h)
	=
	\begin{cases}
		-1, & \text{if } j=1,\\
		0, & \text{if } j >1,
	\end{cases}
\end{equation}
and, thus, by \eqref{eq:Aklp},
\begin{align} \label{eq:1121a}
	\prod_{p \nmid h}
	&\left(
	1
	+
	\sum_{j=1}^\infty
	\frac{c_{p^j}(h)}{p^{2j}}
	A_{k,\ell}(p^j)
	\right)
	=
	\prod_{p \nmid h}
	\left(
	1
	-
	\frac{1}{p^2}
	A_{k,\ell}(p)
	\right)
	\\&
	=
	\prod_{p \nmid h}
	\left(
	1
	-
	\left(
	1 - 
	\left(1- \frac{1}{p}\right)^{k-1}
	\right)
	\left(
	1 - 
	\left(1- \frac{1}{p}\right)^{\ell-1}
	\right)
	\right).
\end{align}

On the other hand, if $p \mid h$, then
\begin{equation}
	c_{p^j}(h)
	=
	\begin{cases}
		p^j-p^{j-1},& \text{if } j \le \nu_p(h),\\
		-p^{\nu_p(h)},& \text{if } j=\nu_p(h)+1,\\
		0,& \text{if } j >\nu_p(h)+1.
	\end{cases}
\end{equation}
Thus, by the above,
\begin{align} \label{eq:1121b}
	&\prod_{p \mid h}
	\left(
	1
	+
	\sum_{j=1}^\infty
	c_{p^j}(h)
	\frac{A_{k,\ell}(p^j)}{p^{2j}}
	\right)
	= \prod_{p \mid h}
	\left(
	1
	+ \sum_{1\le j \le \nu_p(h)}
	(p^j - p^{j-1})
	\frac{A_{k,\ell}(p^j)}{p^{2j}}
	-
	p^{\nu_p(h)}
	\frac{A_{k,\ell}(p^{\nu_p(h)+1})}{p^{2(\nu_p(h)+1)}}
	\right).
\end{align}

Hence, by \eqref{eq:1121a} and \eqref{eq:1121b}, 
\begin{equation} \label{eq:1125}
	B_{k ,\ell}(h)
	= 
	C_{k, \ell}
	f_{k,\ell}(h).
\end{equation}
This completes the heuristic computation.

\begin{rmk}
	Equations \eqref{eq:712}, \eqref{eq:718}, and \eqref{eq:1125} provide a new arithmetic interpretation of the local factor: One has
	\begin{equation}
		C_{k, \ell}
		f_{k,\ell}(h)
		=
		\sum_{q=1}^\infty
		\frac{c_q(h)}{q^2}
		A_{k,\ell}(q),
	\end{equation}
	where $A_{k,\ell}(q)$ are certain normalized RF coefficients of $\tau_k(n)$.
\end{rmk}
\begin{rmk}
	The operations ``replacing $(-1)^k \log^{k-1}(q)$ by $\log^{k-1} X$" and interchanging order of summations of conditionally convergent series are obviously not rigorous, and could be even wrong. But we do them because they produce the right answer.
\end{rmk}

\subsection{Verifying (\ref{eq:main})}
We match our predictions \eqref{eq:C} and \eqref{eq:f} with the following forms from Ng and Thom \cite{NgThom2019}[equations (1.6) \& (1.27)] who had previously computed these quantities.

\begin{lemma}[\cite{NgThom2019} (1.6) p. 98]
	\begin{equation}\label{eq:NgThomC}
		C_{k, \ell}
		=
		\prod_p
		\left(
		\left(
		1 - \frac{1}{p}
		\right)^{k-1}
		+ 
		\left(
		1 - \frac{1}{p}
		\right)^{\ell-1}
		-
		\left(
		1 - \frac{1}{p}
		\right)^{k+\ell-2}
		\right).
	\end{equation}
\end{lemma}

For the local factor $f_{k, \ell}$, Ng and Thom gave five equivalent expressions \cite{NgThom2019}[c.f. (1.7), (1.8), (1.27), (1.28), (4.6)] for this quantity. We chose the following expression to make the verification seemingly child's play.
\begin{lemma}[\cite{NgThom2019} (1.27) p. 102]
	\begin{equation} \label{eq:NgThomf}
		f_{k, \ell}(p^\alpha)
		=
		\dfrac{\sum_{j=0}^{\alpha} \left( \frac{\sigma_{k-1}(p^j, 1) \sigma_{\ell - 1}(p^j, 1)}{p^j} - \frac{\sigma_{k-1}(p^{j+1}, 1) \sigma_{\ell - 1}(p^{j+1}, 1)}{p^{j+2}} \right) }{ \left( 1 - \frac{1}{p} \right)^{k-1} + \left( 1 - \frac{1}{p} \right)^{\ell-1} - \left( 1 - \frac{1}{p} \right)^{k+\ell-2}},
	\end{equation}
	where
	\begin{equation} \label{eq:sigmak}
		\sigma_{k}(p^{j}, s)
		=
		\dfrac{\sum_{i=0}^{\infty} \frac{\tau_{k}(p^{j+i})}{p^{is}}}{ \sum_{i=0}^{\infty} \frac{\tau_{k}(p^{i})}{p^{is}} }
		=
		\left( 1 - \frac{1}{p} \right)^k
		\sum_{i=0}^{\infty} \frac{\tau_{k}(p^{j+i})}{p^{is}}.
	\end{equation}
\end{lemma}

We now come to the sole theorem of this note.
\begin{thm}\label{thm:sole}
	The quantity $C_{k,\ell}$ in \eqref{eq:C} matches with \eqref{eq:NgThomC}, and the same is true for $f_{k, \ell}$ from \eqref{eq:f} and \eqref{eq:NgThomf} .
\end{thm}

\begin{proof}
It is clear that $C_{k,\ell}$ in \eqref{eq:C} is exactly the same as \eqref{eq:NgThomC}. The function $f_{k, \ell}(h)$ in \eqref{eq:f} is multiplicative in $h$, so it suffices to check for prime powers $h=p^\alpha$, where $\alpha = \nu(h)$. The factors appearing in the denominators of both \eqref{eq:f} and \eqref{eq:NgThomf} are identical. This observation motivated the choice \eqref{eq:NgThomf}. The factors $A_{k, \ell}$ and $\sigma_k$ in the numerators are related via the following identity.
\begin{lemma}
	We have
	\begin{equation} \label{eq:756}
		A_{k, \ell}(p^j)
		=
		\sigma_{k-1}(p^j, 1)
		\sigma_{\ell-1}(p^j, 1),
	\end{equation}
	where $A_{k, \ell}(p^j)$ and $\sigma_{k-1}(p^j, 1)$ are given in \eqref{eq:Akl} and \eqref{eq:sigmak}, respectively.
\end{lemma}
\begin{proof}
	This identity follows from a simple change of variables and the identity
	\begin{equation}
		\tau_{k}(p^j)
		= 
		\binom{k + j - 1}{k-1}.
	\end{equation}
	As such,
	\begin{equation}
		\sum_{ \alpha \ge j}
		\binom{k + \alpha -2}{k-2}
		p^{j - \alpha}
		=
		\sum_{i=0}^\infty
		\dfrac{\tau_{k-1}(p^{j+i})}{p^i}.
	\end{equation}
	Thus,
	\begin{equation}
		A_{k, \ell}(p^j)
		= 
		\left( 1 - \frac{1}{p} \right)^{k-1}
		\sum_{i=0}^{\infty} \frac{\tau_{k-1}(p^{j+i})}{p^{is}}
		\left( 1 - \frac{1}{p} \right)^{\ell-1}
		\sum_{i=0}^{\infty} \frac{\tau_{\ell-1}(p^{j+i})}{p^{is}}
		=
		\sigma_{k-1}(p^j, 1)
		\sigma_{\ell-1}(p^j, 1).
	\end{equation}
\end{proof}
Next, the $j=0$ term of the first term in the numerator of \eqref{eq:NgThomf} is equal to
\begin{equation}
	\frac{\sigma_{k-1}(1, 1) \sigma_{\ell - 1}(1, 1)}{1}
	= 1.
\end{equation}
The sum from $j=1$ to $\alpha$ of the first term in the numerator of \eqref{eq:NgThomf} is equal to, by \eqref{eq:756},
\begin{equation}
	\sum_{j=1}^{\alpha} 
	\frac{\sigma_{k-1}(p^j, 1) \sigma_{\ell - 1}(p^j, 1)}{p^j}
	=
	\displaystyle\sum_{1\le j \le \nu_p(h)}
	p^j
	\frac{A_{k,\ell}(p^j)}{p^{2j}}.
\end{equation}

The $j=\alpha$ term of the second term in the numerator of \eqref{eq:NgThomf} is equal to
\begin{equation}
	- \frac{\sigma_{k-1}(p^{\alpha+1}, 1) \sigma_{\ell - 1}(p^{\alpha+1}, 1)}{p^{\alpha+2}}
	=
	- \frac{A_{k, \ell}(p^{\alpha +1})}{p^{\alpha + 2}}
	=
	-
	p^{\nu_p(h)}
	\frac{A_{k,\ell}(p^{\nu_p(h)+1})}{p^{2(\nu_p(h)+1)}},
\end{equation}
by \eqref{eq:756} and $\alpha = \nu_p(h)$. The sum from $j=0$ to $\alpha - 1$ of the second term in the numerator of \eqref{eq:NgThomf} is equal to
\begin{equation}
	- \sum_{j=0}^{\alpha - 1} \frac{\sigma_{k-1}(p^{j+1}, 1) \sigma_{\ell - 1}(p^{j+1}, 1)}{p^{j+2}}
	=
	- \displaystyle\sum_{1\le j \le \nu_p(h)}
	p^{j-1}
	\frac{A_{k,\ell}(p^j)}{p^{2j}},
\end{equation}
after another change of variables and application of \eqref{eq:756}. Thus, by the above four identities,
\begin{align}
	&\sum_{j=0}^{\alpha} \left( \frac{\sigma_{k-1}(p^j, 1) \sigma_{\ell - 1}(p^j, 1)}{p^j} - \frac{\sigma_{k-1}(p^{j+1}, 1) \sigma_{\ell - 1}(p^{j+1}, 1)}{p^{j+2}} \right)
	\\& \quad =
	1
	+ \displaystyle\sum_{1\le j \le \nu_p(h)}
	(p^j - p^{j-1})
	\frac{A_{k,\ell}(p^j)}{p^{2j}}
	-
	p^{\nu_p(h)}
	\frac{A_{k,\ell}(p^{\nu_p(h)+1})}{p^{2(\nu_p(h)+1)}}.
\end{align}
Thus, $f_{k, \ell}(p^{\nu_p(h)})$ with $f_{k, \ell}$ in \eqref{eq:f} is equal $f_{k, \ell}(p^\alpha)$ in \eqref{eq:NgThomf}.
\end{proof}

\subsection{Limitations and extensions}
It seems that the theory of RF expansions is unable to ``see" the powers of $\log$'s in \eqref{eq:main}, and so these factors need to be put in by hand. This is somewhat, though not entirely, analogous to the situation where random matrix theory can predict parts of the leading order main terms in moments and ratios of $L$-functions, and the remaining factors need to be inserted separately.

For an extension of the method, it would be very interesting to compute/predict the leading order main term of the triple convolution
\begin{equation}
	\sum_{n \le X}
	\tau_{k_1}(n)
	\tau_{k_2}(n + h_1)
	\tau_{k_3}(n + h_2).
\end{equation}

\bibliographystyle{plain}

\begin{thebibliography}{10}
	\bibitem{Carmichael1932}
	R. D. Carmichael,
	Expansions of arithmetical functions in infinite series,
	Proceedings of London Math. Soc., Ser. 2, Vol. 34., No. 1849 (1932), 1-26.
	
	\bibitem{GMP2014}
	H. Gopalakrishna Gadiyar, M. Ram Murty, R. Padma,
	Ramanujan-Fourier series and a theorem of Ingham,
	 Indian J Pure Appl Math 45, 691–706 (2014). https://doi.org/10.1007/s13226-014-0084-5
	
	\bibitem{Lucht1995}
	L. Lucht,
	Weighted relationship theorems and Ramanujan expansions,
	ACTA ARITHMETICA, LXX.1 (1995), 25-42.
	
	\bibitem{NgThom2019}
	N. Ng, M. Thom,
	Bounds and conjectures for additive divisor sums, Functiones et Approximatio, 60:1 (2019), 97-142.
\end{thebibliography}

\end{document}